\documentclass[12pt,reqno]{amsart}
\usepackage{color}

\usepackage{amsmath}

\usepackage{amsthm}
\usepackage{amssymb}
\usepackage{graphics}
\usepackage{latexsym}
\usepackage{comment}
\usepackage{ulem}

\numberwithin{equation}{section}
\newtheorem{thm}{Theorem}[section]
\newtheorem{prop}[thm]{Proposition}
\newtheorem{lem}[thm]{Lemma}

\theoremstyle{remark}
\newtheorem{rem}{Remark}[section]
\newtheorem{defn}{Definition}

\newcommand{\BBB}{\mathbb}
\newcommand{\R}{{\BBB R}}
\newcommand{\Z}{{\BBB Z}}
\newcommand{\T}{{\BBB T}}

\newcommand{\N}{{\BBB N}}
\newcommand{\C}{{\BBB C}}
\newcommand{\LR}[1]{{\langle {#1} \rangle }}

\newcommand{\al}{\alpha}
\newcommand{\be}{\beta}
\newcommand{\ga}{\gamma}

\newcommand{\vp}{\varphi}
\newcommand{\eps}{\varepsilon}
\newcommand{\e}{\varepsilon}
\newcommand{\ta}{\tau}

\newcommand{\p}{\partial}
\newcommand{\la}{\lambda}

\newcommand{\de}{\delta}

\newcommand{\Om}{\Omega}

\newcommand{\supp}{\operatorname{supp}}

\newcommand{\I}{\infty}

\newcommand{\EQS}[1]{\begin{align} #1 \end{align}}
\newcommand{\EQQS}[1]{\begin{align*} #1 \end{align*}}

\newcommand{\1}{{\mathbf 1}}

  \title[LWP for modulated NLS]
  {Remark on the local well-posedness for NLS with the modulated dispersion}

  \author[T. Tanaka]
  {Tomoyuki Tanaka}

  \address{
  Tomoyuki Tanaka\\
  Faculty of Science and Engineering\\
  Doshisha University\\
  Kyotanabe\\
  Kyoto\\
  610-0394 Japan
  }
  \email{tomtanak@mail.doshisha.ac.jp}

\begin{document}
\setcounter{page}{001}

\begin{abstract}
  We consider the Cauchy problem of the nonlinear Schr\"odinger equation with the modulated dispersion and power type nonlinearities in any spatial dimensions.
  We adapt the Young integral theory developed by Chouk-Gubinelli \cite{CG1} and multilinear estimates which are based on divisor counting, and show the local well-posedness.
  Our result generalizes the result by Chouk-Gubinelli \cite{CG1}.
\end{abstract}

\maketitle

\section{Introduction}\label{sec_intro}

We study the nonlinear Schr\"odinger equation with the modulated dispersion:

\EQS{\label{eq1}
  \begin{cases}
    i \p_t u = \Delta u \displaystyle{\frac{d w}{dt}} +|u|^{2k}u,\quad (t,x)\in\R\times\T^d,\\
    		u(0,x)=u_0(x),
  \end{cases}
}
where $k,d\in\N$, $w:\R\to \R$ is a real valued function
and $u:\R\times \T^d\to \C$ is the unknown function.
The goal of this note is to show the local well-posedness for the Cauchy problem of \eqref{eq1}.
Equations of this type arise in the study of an optical fiber with dispersion management \cite{A}.
One of our motivations comes from the mathematical study of the nonlinear Schr\"odinger equation with the modulated dispersion:
\EQS{\label{eq_BM}
  i du + \p_x^2 u \circ d\be +|u|^{2} u\, dt=0,\quad x\in \R,\ t>0,
}
where $\be$ is a standard real valued Brownian motion and the product is a Stratonovich product \cite{dBD1}.
The equation \eqref{eq_BM} is obtained as a limit (as $\e\to0$) of the following equation
\EQQS{
  i\p_t u+\frac{1}{\e}m\bigg(\frac{t}{\e^2}\bigg)\p_x^2 u  + |u|^2 u=0,
  \quad x\in \R,\quad t>0,
}
where $m$ satisfies certain assumptions \cite{M06}.
When $w(t)=\be(t)$ is a Brownian motion, the derivative of $w$ is not well-defined.
So, the equation of \eqref{eq1} needs to be interpreted in a proper sense as in \eqref{eq_BM}.
The mathematical analysis of dispersion managed models has attracted much attention in recent years.
See \cite{dBD1,CG21,CG2,DT1,HL06,Ste23,ZGTJH}.
In \cite{CG1}, Chouk-Gubinelli proposed a new approach to study nonlinear dispersive equations by establishing the theory of Young integral.
In particular, they showed the local well-posedness for \eqref{eq1} with $(d,k)=(1,1)$ in $L^2(\T)$ when $w$ is $\rho$--irregular for $\rho>\frac{1}{2}$ (see Definition \ref{def_irre}).
The goal of this article is to generalize their result concerning the dimension and the order of the nonlinearity.

When $w(t)=t$, the equation \eqref{eq1} becomes the nonlinear Schr\"odinger equation
\EQQS{
  i \p_t u = \Delta u  +|u|^{2k}u,
}
and the Cauchy problem
has been estensively studied since the pioneering work by Bourgain \cite{B93}.
The local well-posedness in $H^s(\T^d)$ is known to holds for any $d,k\in\N$ and any subcritical/ctirical regularities $s\ge s_c=\frac{d}{2}-\frac{1}{k}$ with the exception of the case $(d,k)=(1,1)$.
In this case, the Cauhcy problem is globally well-posed in $L^2(\T)$ by \cite{B93}.
See also \cite{Tao}.
There is also a large body of literature dealing with random data nonlinear dispersive equations \cite{BT08,BT08-2,DNY3}.

In this note, we do not restrict ourselves to the case where $w$ is a Brownian motion.
Following the same spirit as \cite{CG1}, we consider a function $w$ belonging to some class which is introduced in Definition \ref{def_irre}.
Remark that our approach is purely deterministic.

First we translate the equation of \eqref{eq1} in the Duhamel form.
We denote the unitary group associated to the linear part of $\Delta u\frac{dw}{dt}$ by $U^w(t)$, i.e.,
\EQQS{
  U^w(t) \vp
  = \sum_{n\in\Z^d} e^{in\cdot x-i|n|^2 w(t)}\hat{\vp}(n),
}
where $\hat{\vp}$ is the Fourier transform of $\vp$.
Notice that $w(-t)=-w(t)$ does not hold in general and we have $(U^w(t))^{-1}=U^{-w}(t)$ instead.
Then, the Duhamel formula corresponding to \eqref{eq1} is written as
\EQS{\label{duhamel1}
  	\vp(t) = U^w(t) \vp_0
  	 -iU^w(t) \int_0^t U^{-w}(\ta)(|\vp(\tau)|^{2k}\vp(\tau))d\tau.
}

We will investigate the effect by $w$ to the regularity of the well-posedness of \eqref{duhamel1}.
For that purpose, following \cite{CG1}, we introduce $(\rho,\ga)$--irregularity of $w$.

\begin{defn}\label{def_irre}
	Let $\rho>0$ and $0<\ga\le 1$. We say that a function $w\in C([0,T];\R)$ is $(\rho,\ga)$--irregular if:
	\EQS{
    \|\Phi^w\|_{\mathcal{W}_T^{\rho,\ga}}
		:=\sup_{a\in\R}\sup_{0\le s<t\le T} (1+|a|)^\rho
		\frac{|\Phi_t^w(a)-\Phi_s^w(a)|}{|t-s|^\ga}<\infty,
  }
	where $\Phi_t^w(a)=\int_0^t e^{iaw(r)}dr$.
	Moreover, we say that $w$ is $\rho$--irregular if there exists $\ga>\frac 12$ such that $w$ is $(\rho,\ga)$--irregular.
\end{defn}

\begin{rem}
	If $w$ is a continuous function on $[0,T]$, then we see from the mean value theorem that $\|\Phi^w\|_{\mathcal{W}_T^{0,\ga}}<\infty$ for $0<\ga\le 1$.
\end{rem}

\begin{rem}
  It is well-known that $\sup_{s<t}|t-s|^{-1-\e}|f(t)-f(s)|<\I$ for a continuous function $f$ implies that $f$ is constant.
  However, $\Phi_t^w(a)$ cannot be constant with respect to $t$ because of the definition of $\Phi_t^w(a)$.
  This is why we do not consider the case $\ga>1$.
\end{rem}

\begin{rem}
	The purely dispersive case corresponds to $w(t)=t$ for $t\in\R$.
	In this case, $w$ is $(\rho,\ga)$--irregular with $\rho+\ga\le 1$.
	Indeed, we have
	\begin{align*}
		|\Phi_t^w(a)-\Phi_s^w(a)|
		\le
		\begin{cases}
			|a|^{\theta-1}|t-s|^\theta, &{\rm if}\quad |a|\ge 1,\\
			|t-s|, &{\rm if}\quad |a|\le 1
		\end{cases}
	\end{align*}
	for any $\theta\in[0,1]$.
	It then follows from the definition that
	\begin{align*}
		\|\Phi^w\|_{\mathcal{W}_T^{\rho,\ga}}
		\lesssim  \sup_{|a|\le 1}\sup_{0\le s<t\le T}|t-s|^{1-\ga}
		 +\sup_{|a|> 1}\sup_{0\le s<t\le T}|a|^{\rho+\theta-1}
		  |t-s|^{\theta-\ga}<\infty,
	\end{align*}
	provided that $\rho+\ga\le 1$.
\end{rem}

It is worth recalling that a fractional Brownian motion is $(\rho,\ga)$--irregular for some $\rho,\ga$.
The following theorem is obtained in \cite{CateGubi}.

\begin{thm}
	Let $(W_t)_{t\ge 0}$ be a fractional Brownian motion of Hurst index $H\in(0,1)$ then for any $\rho<\frac{1}{2H}$ there exist $\ga>\frac 12$ so that with probability one the sample paths of $W$ are $(\rho,\ga)$--irregular.
\end{thm}

In this article, we construct solutions by using the Young integral theory developed by \cite{CG1}.
First we define
\EQQS{
  \mathcal{N}(\vp_1,\cdots, \vp_{2k+1})
  :=\bigg(\prod_{j=1}^{k+1}\vp_{2j-1}\bigg)\prod_{l=1}^{k}\bar{\vp}_{2l}.
}
For simplicity, we denote $\mathcal{N}(\vp,\cdots,\vp)$ by $\mathcal{N}(\vp)$.
We also define a map $X:[0,\infty)\times H^{s_1}(\T^d)\times\cdots\times H^{s_1}(\T^d)\to H^{s_2}(\T^d)$ by
\EQS{\label{def_X2}
    X_t(\vp_1,\cdots,\vp_{2k+1})
    =-i \int_0^t U^{-w}(\ta)\mathcal{N}(U^{w}(\tau)\vp_1,\cdots, U^{w}(\tau)\vp_{2k+1}) d\tau
}
for some $s_1,s_2\ge 0$ (for example, $s_1>d/2$ and $s_2\ge 0$).
For simplicity, we write $X_t(\vp)=X_t(\vp,\bar{\vp},\cdots,\bar{\vp},\vp)$ and we also
we write $X_{s;t}=X_t-X_s$.
Then, thanks to Proposition \ref{prop_young3},
when $w$ is $(\rho,\ga)$--irregular, it holds that $X\in C^\ga([0,T];\mathcal{L}(H^s(\T^d)))$ for $T>0$ and $s>\frac d2- \frac{\rho}{k}$.
See Subsection \ref{ssec_notation} for the definition $\mathcal{L}(V)$.
Therefore, we can define the Young integral associated to this $X$ as the limit of suitable Riemann sums (see Theorem \ref{thm_young}).
For that purpose, let $\{t_j\}_{j=0}^n\subset\R$ satisfying $0=t_0<t_1<\cdots <t_{n-1}<t_n=t$ and set $\de_t:=\max_{1\le j\le n}|t_j-t_{j-1}|$.
We define
\EQS{\label{eq_young2}
  	\int_0^t X_{d\tau} (g_1(\tau),\cdots,g_{2k+1}(\ta))
    :=\lim_{\de_t\to0}
  	 \sum_{j=0}^n X_{t_j;t_{j+1}}(g_1(t_j),\cdots,g_{2k+1}(t_j)).
}
Under these preparations, we will solve the following equation:
\begin{align}\label{duhamel2}
	\varphi(t)=\varphi_0+\int_0^t X_{d\tau}(\varphi(\ta)).
\end{align}
We can see that $\psi(t) = U^w(t) \vp(t)$ solves \eqref{duhamel1} if $\vp$ is smooth and satisfies \eqref{duhamel2}.
Moreover, thanks to Theorem \ref{thm_young}, $\psi$ satisfies \eqref{eq1} if $w\in C^1$.
Our main result is the following:


\begin{thm}\label{thm_main}
	Let $p,d\in \N$ with $(d,k)\neq(1,1)$ and let $T_0>0$.
	Let $0<\rho\le 1$.
	Let $0<\la< \ga\le 1$ and $\ga+\la>1$.
	Assume that a real-valued continuous function $w$ is $(\rho,\ga)$--irregular.
	Then for any $s> s(\rho):= \frac d2-\frac{\rho}{k}$ and any $\psi\in H^s(\T^d)$, there exists $T=T(\|\psi\|_{H^s},\|\Phi^w\|_{\mathcal{W}_{T_0}^{\rho,\ga}})\in(0,T_0)$ such that there exists a solution $u\in C^\la([0,T];H^s(\T^d))$ to \eqref{duhamel2} emanating from $\psi$.
	Moreover, $u$ is the unique solution to \eqref{duhamel2} associated with $\psi$ that belongs to $C^\la([0,T];H^s(\T^d))$.
	Finally, the solution map $\psi\to u$ is continuous from the ball of $H^s(\T^d)$ with radius $R$ centered at the origin into $C^\la([0,T];H^s(\T^d))$.
\end{thm}

\begin{rem}
  We make some remarks with respect to the regularity.
	\begin{enumerate}
    \item Recall that $s(1)=\frac{d}{2}-\frac{1}{k}$ is the scaling ctitical index for following nonlinear Schr\"odinger equation:
  	\begin{align*}
  		i\p_t u=\Delta u+|u|^{2k}u.
  	\end{align*}
  	So, when $\rho=1$, we can construct the solution at the subcritical range, i.e., $s> s(1)$ with $(d,k)\neq(1,1)$.
    \item When $\rho>1$, we can still prove the local well-posedness in $H^s(\T^d)$ for $s>s(1)$ with $(p,d)\neq (1,1)$ by the same proof as that of Theorem \ref{thm_main}.

  	\item Although our theorem does not contain the case $(d,k)=(1,1)$, we can prove the local well-posedness in this case in $H^s(\T)$ for $s>0$ and $\rho>\frac 12$.
  	Recall that Chouk-Gubinelli's result \cite{CG1} is $H^s(\T)$ for $s\ge 0$.

  	\item Notice that $w(t)=1$ satisfies $\|\Phi^w\|_{\mathcal{W}_T^{0,\ga}}<\infty$ for $0<\ga\le 1$.
  	The corresponding equations is $i\p_t u=|u|^{2k}u$.
  	By a classical argument, particularly using the Sobolev inequality $L^\I(\T^d)\hookrightarrow H^{\frac{d}{2}+}(\T^d)$, we can show the well-posedness for this equation in $H^s(\T^d)$ for $s>\frac d2$.
  	This observation links to the case where $\rho>0$ is sufficiently small in Theorem \ref{thm_main}.
  \end{enumerate}
\end{rem}

\begin{rem}
  In Theorem \ref{thm_main}, we first fix $T_0>0$ for the sake of completeness.
  However, we choose $T\in(0,T_0)$ sufficiently small in the proof, so we can assume $T_0=1$ without loss of generality.
\end{rem}

The proof of Theorem \ref{thm_main} is based on the Young integral theory developed by \cite{CateGubi,CG1}.
In \cite{CG1}, Chouk-Gubinelli studied the Cauchy problem of \eqref{eq1} with $(d,k)=(1,1)$, i.e., 1d cubic case, and showed the local well-posedness in $L^2(\T)$ when $w$ is $(\rho,\ga)$--irregular with $\rho>\frac{1}{2}$ and $\ga>\frac{1}{2}$.
They made a use of the irregularity of $w$, and extracted a smoothing property in order to close estimates at low regularities.
This phenomenon is called a regularization by noise, and this was achieved by (divisor) counting estimates for a cubic nonlinearity.
In the present work, we treat a nonlinearity of order $2k+1$.
So, we borrow the multilinear estimate (Lemma \ref{lem_rho2}) proved by \cite{K21} in the study of the unconditional uniqueness of \eqref{eq1} with $w(t)=t$.
We interpolate this estimate with a trivial estimate (that is, a dispersion-less estimate) and then we apply Theorem \ref{thm_young2}.
For more connections with stochastic analysis, see Section 2 of \cite{Rob23}.

This paper is organized as follows.
In Section 2, we recall some results of Young integrals.
In Section 3, we show the multilinear estimate which is the main estimate of this article.
In Section 4, we give a proof of Theorem \ref{thm_main}.

\subsection{Notation}\label{ssec_notation}
In this subsection, we fix the notation.
We write $\LR{\cdot}=(1+|\cdot|^2)^{\frac{1}{2}}$.
We denote weighted sequential $\ell^p$-norms by $\|\psi\|_{\ell_s^p}:=\|\LR{\cdot}^s \psi(\cdot)\|_{\ell^p(\Z^d)}$ for $p\in[1,\infty]$ and $s\in\R$.
In what follows, we only consider Hilbert spaces $V$ whose norm satisfies $\|\bar{v}\|_{V}=\|v\|_V$ for $v\in V$.
We denote the space of all $n$--linear\footnote{In this article, we consider conjugate linear operators as linear operators. Here, if $T$ is a conjugate linear operator, then $T(\al u+\be v)=\bar{\al}T(u)+\bar{\be}T(v)$ for $u,v\in D(T)$ and $\al,\be\in\C$.} bounded operators on $V\times\cdots\times V$ (i.e., $n$ variables) with values in $W$ by $\mathcal{L}_n(V,W)$.
Its norm is defined by
\EQQS{
  \|f\|_{\mathcal{L}_n(V,W)}
	=\sup_{\substack{\psi_1,\dots,\psi_n\in V\\\psi_j\neq0}}
	\frac{\|f(\psi_1,\dots,\psi_n)\|_W}{\|\psi_1\|_V \cdots \|\psi_n\|_V}
}
and set $\mathcal{L}_n(V):=\mathcal{L}_n(V,V)$.
Clearly we have $\mathcal{N}\in \mathcal{L}_{2k+1}(H^{s_1},H^{s_2})$ for some $s_1,s_2\ge 0$.
Let $T>0$ and $U$ be a Banach space.
We define $C^\ga([0,T];U)$ as the space of $\ga$--H\"older continuous functions from $[0,T]$ to $U$.
We also define the semi-norm
\EQQS{
  \|f\|_{C^\ga([0,T];U)}
	:=\sup_{0\le s< t\le T}\frac{\|f(t)-f(s)\|_U}{|t-s|^\ga}.
}
We equipe the space $C^\ga([0,T];U)$ with the following norm:
\EQQS{
  \|f\|_{C^{0,\ga}([0,T];U)}
	:=\|f\|_{C([0,T];U)}+\|f\|_{C^\ga([0,T];U)}.
}
Recall that the space $C^\ga([0,T];U)$ is a Banach space with the norm $\|\cdot\|_{C^{0,\ga}([0,T];U)}$.
We write $\|\cdot\|_{C_T^{0,\ga}U}:=\|\cdot\|_{C^{0,\ga}([0,T];U)}$ for simplicity.

\section{Young Integral}

Our local well-posedness result depends on the following result which is proved in [Theorem 2.3, \cite{CG1}].

\begin{thm}[Young integral]\label{thm_young}
	Let $T,\ga,\la>0$ and let $n\in\N$.
	Assume that $\ga+\la>1$.
	Let $f\in C^{\ga}([0,T];\mathcal{L}_n(V))$ and $\{g_j\}_{j=1}^n\subset C^{\la}([0,T];V)$ then the limit of Riemann sums
	\EQQS{
    I_t(f;g_1,\cdots,g_n)
		&=\int_0^t f_{d\ta} (g_1(\ta),\cdots,g_n(\ta))\\
		&:=\lim_{\de_t\to 0}\sum_{i}(f_{t_{i+1}}(g_1(t_i),\cdots,g_n(t_i))-f_{t_{i}}(g_1(t_i),\cdots,g_n(t_i)))
  }
	exist in $V$ as the partition $\{t_i\}_i$ of $[0,t]$ is refined,
	it is independent of the partition, and we have
	\begin{align*}
		&\|I_t(f;g_1,\cdots,g_n)-I_s(f;g_1,\cdots,g_n)-(f_t-f_s)(g_1(s),\cdots,g_n(s))\|_V\\
		&\le (1-2^{1-\ga-\la})^{-1}|t-s|^{\ga+\la}
		\|f\|_{C_T^{\ga}(\mathcal{L}_n(V))}
    \sum_{j=1}^n \|g_j\|_{C_T^\la V}
    \prod_{\substack{l=1,\\ j\neq j}}^n \|g_l\|_{C_T V}.
	\end{align*}
\end{thm}

In what follows, we use the shorthand notation $f_t(\psi)=f_t(\psi,\cdots,\psi)$.
In this note, we will take $f=X$ and show that $X\in C^\ga([0,T];\mathcal{L}_{n}(H^s(\T^d)))$ in Proposition \ref{prop_young3}, where $X$ is defined in \eqref{def_X2}.

The following theorem is essentially proved in [Theorem 2.4, \cite{CG1}], and provides a fixed point procedure in the context of the Young integral.
For more general theory, see \cite{Gal23}.

\begin{thm}[Young solutions]\label{thm_young2}
	Let $n\in\N$ and $T_0>0$.
	Let $\ga,\la>0$ satisfy $\la< \ga\le 1$ and $\ga+\la>1$.
	Assume that $Y\in C^\ga([0,T_0];\mathcal{L}_n(V))$ and $X_0\equiv 0$.
	For any $\psi_0\in V$ there exists $T\in (0,T_0)$ depending only on $\|Y\|_{C^\ga ([0,T_0]; \mathcal{L}_n(V))}$ and $\|\psi_0\|_V$ such that the Young equation associated to the operator $Y$ given by
	\EQS{
    \label{eq_young}
		\psi(t)=\psi_0+\int_0^t Y_{d\tau }(\psi(\tau)),\quad 0\le t\le T
  }
	has a unique solution $\psi\in C^\la([0,T];V)$.
  Moreover, the solution map $\psi_0\to \psi$ is continuous from the ball of $H^s(\T^d)$ with radius $R$ centered at the origin into $C^\la([0,T];H^s(\T^d))$.
\end{thm}

\begin{proof}
	Let $T\in (0,T_0)$ and we define $\{\psi_m\}_{m\in\N_0}\subset C^{\la}([0,T];V)$ as $\psi_m(0)=\psi_0$.
  We also define
	\EQQS{
    \psi_{m+1}(t)
		=\psi_0+\int_0^t Y_{d\ta}(\psi_m(\ta)),\quad 0\le t\le T.
  }
	First we show that $\psi_m\in C^\la ([0,T];V)$ by induction.
	By definition, we have
	\EQQS{
    \psi_1(t)=\psi_0+\int_0^t Y_{d\ta}(\psi_0)
		=\psi_0+Y_t(\psi_0),
  }
	which implies that $\psi_1\in C^\la([0,T];V)$ since
	$\|\psi_1\|_{C_T^\la V}\lesssim T^{\ga-\la}\|Y\|_{C_T^\ga (\mathcal{L}_n(V))}\|\psi_0\|_{V}^n$ and $\ga\ge\la$.
	In particular, $\psi_1(t)$ is continuous in $V$.
	By the triangle inequality, we also have $\|\psi_m\|_{C_T V}\le T^{\la}\|\psi_m\|_{C_T^\la V}+\|\psi_0\|_V$.
	We see from Proposition \ref{thm_young} that
	\EQQS{
    &\bigg\|\int_s^t Y_{d\tau}(\psi_m(\tau))\bigg\|_{V}\\
		&\le \bigg\|\int_s^t Y_{d\tau}(\psi_m(\tau))-(Y_t-Y_s)(\psi_m(s))\bigg\|_{V}
		 +\|(Y_t-Y_s)(\psi_m(s))\|_V\\
		&\lesssim  |t-s|^{\la+\ga} \|Y\|_{C_T^\ga(\mathcal{L}_n(V))}\|\psi_m\|_{C_T^\la V}
		\|\psi_m\|_{C_T V}^{n-1}
		+\|Y_{s;t}\|_{\mathcal{L}_n(V)}\|\psi_m(s)\|_V^n\\
		&\lesssim  |t-s|^{\ga} \|Y\|_{C_T^\ga(\mathcal{L}_n(V))}
		(\|\psi_0\|_{V}+T^\la\|\psi_m\|_{C_T^\la V})^n,
  }
	which implies that $\psi_{m+1}\in C^\la([0,T];V)$ by the induction on $m$.
	It immediately follows that $\psi_{m+1}\in C([0,T];V)$ and
	\EQQS{
    \|\psi_{m+1}\|_{C_T^\la V}
		\lesssim  T^{\ga-\la}\|Y\|_{C_T^\ga(\mathcal{L}_n(V))}
		(\|\psi_0\|_{V}+T^\la\|\psi_m\|_{C_T^\la V})^n.
  }
	Next, we show that $\{\psi_m\}$ is Cauchy in $C^\la([0,T];V)$.
	As in the above estimate, we obtain
	\EQQS{
    &\|\psi_{m_1+1}-\psi_{m_2+1}\|_{C_T^{0,\la} V}\\
		&\le CT^{\ga-\la} \|Y\|_{C_T^\ga (\mathcal{L}_n(V))}
		(\|\psi_0\|_V+\|\psi_{m_1}\|_{C_T^\la V}
		 +\|\psi_{m_2}\|_{C_T^\la V})^{n-1}
		 \|\psi_{m_1}-\psi_{m_2}\|_{C_T^{0,\la} V}.
  }
  for $m_1,m_2\in\N$ and $0<T<1$.
  By choosing $T>0$ sufficiently small, the claim follows.
  Moreover, there exists $\psi\in C^\la([0,T];V)$ such that $\|\psi_m-\psi\|_{C_T^{0,\la} V}\to 0$ as $m\to\I$.
  Since $Y\in C^\ga([0,T];\mathcal{L}_n(V))$, we see from the above estimate that
  \EQQS{
    \bigg\|\int_0^t Y_{d\tau}(\psi_m(\tau))
    -\int_0^t Y_{d\tau}(\psi(\tau))\bigg\|_{C_T^{0,\ga}V}
    \to 0
  }
  as $m\to\I$,
  which implies that $\psi$ satisfies \eqref{eq_young} for $0\le t\le T$.
  Notice that the solution is unique in $C^\la([0,T];V)$.
  Finally, following the above argument with a slight modification, we can obtain the continuous dependence.
  This completes the proof.
\end{proof}

\begin{rem}
	In Theorem \ref{thm_young2}, notice that $\la<\ga$ and $\ga+\la>1$ imply $\ga>\frac12$.
	This corresponds to Theorem 2.4 in \cite{CG1}.
\end{rem}


\section{Multilinear estimates}

In this section, we show the multilinear estimate which assures that our map $X$ (defined in \eqref{def_X2}) satisfies the assumption of Theorem \ref{thm_young2}.
We need to have an estimate of the following type:

\begin{prop}\label{prop_multi}
	Let $d,k\in\N$ with $(d,k)\neq(1,1)$ and let $0<\rho\le 1$.
	Let $s>s(\rho)=\frac d2-\frac{\rho}{k}$ and $-s\le s'\le s$.
	Then, for any $1\le q\le 2k+1$
	\begin{align}\label{eq_2.1}
		\bigg\|\sum_{\substack{n_1,\dots,n_{2k+1}\in\Z^d\\n_0-n_1+\dots-n_{2k+1}=0}}
		\frac{ 1 }{\LR{\Om}^\rho }
		 \prod_{j=1}^{2k+1}\psi_j(n_j)\bigg\|_{\ell_{s'}^2(\Z_{n_0}^d)}
		\lesssim  \|\psi_q\|_{\ell_{s'}^2}
		\prod_{\substack{j=1\\j\neq q}}^{2k+1}\|\psi_j\|_{\ell_s^2},
	\end{align}
	where $\Om=|n_0|^2-|n_1|^2+\cdots-|n_{2k+1}|^2$.
\end{prop}


The following estimate corresponds to the case $\rho=0$ in \eqref{eq_2.1}.

\begin{lem}\label{lem_rho3}
	Let $d,k\in\N$.
	Let $s\ge \frac d2$.
	Then, we have
	\begin{align}\label{eq_2.7}
		\sum_{\substack{n_0,n_1,\dots,n_{2k+1}\in\Z^d\\n_0-n_1+\cdots-n_{2k+1}=0}}
		\prod_{j=0}^{2k+1}\psi_j(n_j)
		\lesssim  N_{\mathrm{max}}^{-2s}\prod_{j=0}^{2k+1}N_j^s \|\psi_j\|_{\ell^2(\Z^d)}
	\end{align}
	for any nonnegative functions $\{\psi_j\}_{j=0}^{2k+1}\subset\ell^2(\Z^d)$ satisfying $\supp \psi_j\subset\{n\in\Z^d:N_j\le \LR{n}<2N_l\}$, where $N_{\mathrm{max}}:=\max_{0\le j\le 2k+1}N_j$.
	Here, the implicit constant is uniform in $\{N_j\}$.
\end{lem}

\begin{proof}
	By the symmetry, we may assume that
	\begin{align*}
		N_0\ge N_2\ge \cdots\ge N_{2k},
		\quad
		N_1\ge N_3\ge \cdots \ge N_{2k+1},
		\quad N_0\ge N_1.
	\end{align*}
	Note that $N_{\mathrm{max}}=N_0\sim \max\{N_1,N_2\}$.
	Without loss of generalirty, we may assume that $N_{\mathrm{second}}=N_1$, where $N_{\mathrm{second}}$ is the second largest among $N_0,\cdots,N_{2k+1}$.
	The case $N_{\mathrm{second}}=N_2$ follows similarly: we only have to switch roles of $\psi_1$ and $\psi_2$.
	The Cauchy-Schwarz inequality and the Bernstein inequality show that
	the left-hand side of \eqref{eq_2.7} is bounded by
	\begin{align*}
		&\sum_{n_1,\dots,n_{2k+1}\in\Z^d}\psi_{0}(n_1-n_2+\cdots+n_{2k+1})
		\prod_{j=1}^{2k+1}\psi_j(n_j)\\
		&\le \|\psi_0\|_{\ell^2}\|\psi_1\|_{\ell^2}
		\sum_{n_2,\dots,n_{2k+1}\in\Z^d}\prod_{j=2}^{2k+1}
		\psi_j (n_j)
		\lesssim  \|\psi_0\|_{\ell^2}\|\psi_1\|_{\ell^2}
		\prod_{j=2}^{2k+1} N_j^{\frac d2}\|\psi_j\|_{\ell^2},
	\end{align*}
	which shows \eqref{eq_2.7}.
\end{proof}

The following estimate is due to Kishimoto \cite{K21}, which corresponds to $\rho>1$ in \eqref{eq_2.1}.
Thanks to this strong decay of $\LR{\mu}^{-\rho}$, we can obtain a better estimate than \eqref{eq_2.7} with respect to the regularity.

\begin{lem}[Lemma 3.1, \cite{K21}]\label{lem_rho2}
	Let $d,k\in\N$ with $(d,k)\neq(1,1)$.
	Let $s>s(1)=\frac d2-\frac{1}{k}$.
	Then, we have
	\begin{align}\label{eq_2.6}
		\sum_{\substack{n_0,n_1,\dots,n_{2k+1}\in\Z^d\\n_0-n_1+\cdots-n_{2k+1}=0\\|n_0|^2-|n_1|^2+\cdots-|n_{2k+1}|^2=\mu}}
		\prod_{j=0}^{2k+1}\psi_j(n_j)
		\lesssim  N_{\mathrm{max}}^{-2s}\prod_{j=0}^{2k+1}N_j^s \|\psi_j\|_{\ell^2(\Z^d)}
	\end{align}
	for any $\mu\in\Z$, $\{N_j\}_{j=0}^{2k+1}\subset 2^{\N_0}$, and any nonnegative functions $\{\psi_j\}_{j=0}^{2k+1}\subset\ell^2(\Z^d)$ satisfying $\supp \psi_j\subset\{n\in\Z^d:N_j\le \LR{n}<2N_l\}$, where $N_{\mathrm{max}}:=\max_{0\le j\le 2k+1}N_j$.
	Here, the implicit constant is uniform in $\mu$ and $\{N_j\}$.
\end{lem}

Now, we prove Proposition \ref{prop_multi} by interpolating \eqref{eq_2.7} and \eqref{eq_2.6}.

\begin{proof}[Proof of Proposition \ref{prop_multi}]
	We mainly follow the proof of Corollary 3.2 in \cite{K21}.
	We only consider the case $q=1$.
  Other cases follow from the same argument.
	By duality, if suffices to show
	\EQQS{
    \bigg|\sum_{\substack{n_0,n_1,\dots,n_{2k+1}\in \Z^d\\n_0-n_1+\cdots-n_{2k+1}=0}}
		\frac{1}{\LR{\Om}^\rho}\prod_{j=0}^{2k+1}\psi_j(n_j)\bigg|
		\le C\|\psi_0\|_{\ell_{-s'}^2}\|\psi_1\|_{\ell_{s'}^2}
		\prod_{j=2}^{2k+1}\|\psi_j\|_{\ell_{s}^2},
  }
	where $\Om=|n_0|^2-|n_1|^2+\cdots-|n_{2k+1}|^2$.
	We choose sufficiently small $\e>0$ so that $s=s(\rho)+\e$.
  For this $\e>0$, we also choose $\theta\in(0,1)$ so that $\rho-\theta=k\e$.
  Notice that $\rho/\theta>1$.
	We write $P_N\psi(n):=\1_{\{N\le \LR{n}<2N\}}\psi(n)$.
	By the H\"older inequality in $\mu$, the left-hand side of the above estimate is bounded by
	\begin{align*}
		&\sum_{N_0,\dots,N_{2k+1}\in 2^{\N_0}}\sum_{\mu\in\Z}
		\sum_{n_0,\dots,n_{2k+1}\in\Z^d}
		\frac{\1_{A(\mu)}}{\LR{\mu}^\rho}\prod_{j=0}^{2k+1}|P_{N_j}\psi_j(n_j)|\\
		&= \sum_{N_0,\dots,N_{2k+1}}
		\bigg(\sum_{\mu\in\Z}\frac{1}{\LR{\mu}^{\rho}}
		\bigg(\sum_{n_0,\dots,n_{2k+1}\in\Z^d}\1_{A(\mu)}\prod_{j=0}^{2k+1}|P_{N_j}\psi_j(n_j)|\bigg)^{\theta}\\
		&\quad\quad\quad\times
		\bigg(\sum_{n_0,\dots,n_{2k+1}\in\Z^d}\1_{A(\mu)}\prod_{j=0}^{2k+1}|P_{N_j}\psi_j(n_j)|\bigg)^{1-\theta}\bigg)\\
		&\le \sum_{N_0,\dots,N_{2k+1}}
		 \bigg(\sum_{\mu\in\Z}\frac{1}{\LR{\mu}^{1+}}\sum_{n_0,\dots,n_{2k+1}\in\Z^d}\1_{A(\mu)}\prod_{j=0}^{2k+1}|P_{N_j}\psi_j(n_j)|\bigg)^{\theta}\\
		&\quad\quad\quad \times
 		\bigg(\sum_{\mu\in\Z}\sum_{n_0,\dots,n_{2k+1}\in\Z^d}\1_{A(\mu)}\prod_{j=0}^{2k+1}|P_{N_j}\psi_j(n_j)|\bigg)^{1-\theta}.
	\end{align*}
	Here, $A(\mu)$ is defined by $A(\mu):=\{(n_0,n_1,\dots,n_{2k+1})\in(\Z^d)^{2k+2}:(*)_\mu\}$,
	where
	``$(*)_\mu$" denotes the condition
		\begin{align*}
			n_0-n_1+\cdots-n_{2k+1}=0,\quad
			\Om=|n_0|^2-|n_1|^2+\cdots-|n_{2k+1}|^2=\mu.
		\end{align*}
	We also used $\rho>\theta$.
	Note that $\theta s(1)+\frac{(1-\theta) d}{2}=s(\theta)= s(\rho)+\eps$.
  Notice that the value $\mu=|n_0|^2-|n_1|^2+\cdots-|n_{2k+1}|^2$ is determined once $n_0,\dots,n_{2k+1}$ are given.
	Thus, we have
	\begin{align*}
		\sup_{n_0,\dots,n_{2k+1}\in\Z^d}\sum_{\mu\in\Z}\1_{A(\mu)}=1.
	\end{align*}
	Then, \eqref{eq_2.7} with $s=\frac d2$ and \eqref{eq_2.6} with $s=s(1)$ show that
	\begin{align*}
		&\lesssim  \sum_{N_0,\dots,N_{2k+1}}
		\bigg(\frac{N_0N_1\dots N_{2k+1}}{N_{\mathrm{max}}^2}\bigg)^{s-\e}\prod_{j=0}^{2k+1}\|P_{N_0}\psi_0\|_{\ell^2}\\
		&\lesssim \sum_{N_0,\dots,N_{2k+1}} \bigg(\frac{N_0 N_1\dots N_{2k+1}}{N_{\mathrm{max}}^2}\bigg)^{-\e}
		\|P_{N_0}\psi_0\|_{\ell_{-s'}^2}
		\|P_{N_1}\psi_1\|_{\ell_{s'}^2}
		\prod_{j=2}^{2k+1}\|P_{N_j}\psi_j\|_{\ell_{s}^2}.
	\end{align*}
	At the last inequality, we used $N_0^sN_1^s/N_{\mathrm{max}}^{2s}\le N_0^{-s'}N_1^{s'}$ for any $-s\le s'\le s$.
	For $N_{\mathrm{max}}$ and $N_{\mathrm{second}}\sim N_{\mathrm{max}}$, we can use the Cauchy-Schwarz inequality by orthogonality.
  Here, $N_{\mathrm{second}}$ is the second largest among $N_0,\cdots,N_{2k+1}$.
	On the other hand, summations over other $N_j$'s can be closed thanks to a negative power of $N_j$.
	This completes the proof.
\end{proof}

%
%
%

\section{Well-posedness}

In this section, we prove our main result (Theorem \ref{thm_main}).
For that purpose, we first investigate the property of our map $X$.

\begin{prop}\label{prop_young3}
	Let $k,d\in \N$ with $(k,d)\neq (1,1)$ and $T,\rho,\ga>0$.
	Let $s>\frac d2-\frac{\rho}{k}$.
	Assume that a function $w\in C([0,T];\R)$ is $(\rho,\ga)$--irregular.
	Let $X$
	be defined in \eqref{def_X2}.
	Then, it holds that $X\in C^\ga([0,T];\mathcal{L}_{2k+1}(H^s(\T^d)))$.
\end{prop}

\begin{proof}
	Let $\{\psi_j\}_{j=1}^{2k+1}\subset H^s(\T^d)$.
	We see from Proposition \ref{prop_multi} that
	\EQQS{
    &\|X_{t_1;t_2}(\psi_1,\cdots,\psi_{2k+1})\|_{H^s}\\
		&=\bigg\|\int_{t_1}^{t_2}\sum_{n_1,\dots,n_{2k+1}\in\Z^d}
		 \LR{n}^{s}e^{-iw(\ta)\Om}
		 \1_{\{n=\sum_{j=1}^{2k+1}\zeta_j n_j\}}
     \prod_{j=1}^{2k+1}J_j\hat{\psi}_{j}(n_{j})
		 d\ta\bigg\|_{\ell^2(\Z_n^d)}\\
		&\lesssim  \|\Phi^w\|_{\mathcal{W}_T^{\rho,\ga}}|t_2-t_1|^\ga
		 \bigg\|\sum_{n_1,\dots,n_{2k+1}\in\Z^d}
 		 \frac{\LR{n}^{s}\1_{\{n=\sum_{j=1}^{2k+1}\zeta_j n_j\}}}{\LR{\Om}^\rho}
 		 \prod_{j=1}^{2k+1}|\hat{\psi}_j(\zeta_j n_{j})|\bigg\|_{\ell^2(\Z_n^d)}\\
		&\lesssim  \|\Phi^w\|_{\mathcal{W}_T^{\rho,\ga}}|t_2-t_1|^\ga
		 \prod_{j=1}^{2k+1}\|\psi_j\|_{H^s},
  }
	where $\Om=|n|^2-\sum_{j=1}^{2k+1}\zeta_j|n_j|^2$, $\zeta_{2j-1}=+$, $\zeta_{2j}=-$, $J_{2j-1}\psi=\psi$ and $J_{2j}\psi=\bar{\psi}$ for a function $\psi$.
	Then, we have
	\begin{align*}
		\|X\|_{C^\ga ([0,T];\mathcal{L}_{2k+1}(H^s))}
		&=\sup_{0\le s<t\le T}\sup_{\psi_j\in H^s}
		 \frac{\|X_{s;t}(\psi_1,\cdots,\psi_{2k+1})\|_{H^s}}{|t-s|^\ga \prod_{j=1}^{2k+1}\|\psi_j\|_{H^s}}
		\lesssim  \|\Phi^w\|_{\mathcal{W}_T^{\rho,\ga}}<\infty
	\end{align*}
	as desired.
	This completes the proof.
\end{proof}

%


Now, we are ready to prove our main result (Theorem \ref{thm_main}).

\begin{proof}[Proof of Theorem \ref{thm_main}]
	Let $X$ be defined in \eqref{def_X2}.
	Proposition \ref{prop_young3} implies that $X\in C^\ga([0,T_0]; \mathcal{L}_{2k+1}(H^s(\T^d)))$ for any $T_0>0$.
	From the definition of $X$, it is clear that $X_0(\psi)=0$ for any $\psi\in H^s(\T^d)$.
	Then, Theorem \ref{thm_young2} shows that there exists $T\in(0,T_0)$ depending only on $\|X\|_{C^\ga([0,T_0];\mathcal{L}_{2k+1}(H^s(\T^d)))}$ and $\|\varphi_0\|_{H^s}$ such that there exists a unique solution $\varphi\in C^\la([0,T];H^s(\T^d))$
	to \eqref{duhamel2}, which completes the proof.
\end{proof}

\section*{Acknowlegdements}

The author is grateful to Professor Yuzhao Wang for sugestiong this problem and fruitiful discussion.
The author was supported by the EPSRC New Investigator Award (grant no. EP/V003178/1) and JSPS KAKENHI Grant Number 23K19019.


\begin{thebibliography}{99}
  \bibitem{A}
  G.P. Agrawal, Nonlinear Fiber Optics, 3rd ed., Academic Press, 2001.

  \bibitem{dBD1}
  A. de Bouard and A. Debussche, {\it The nonlinear Schr\"odinger equation with white noise dispersion}, J. Funct.\ Anal. {\bf 259} (2010), 1300--1321.

  \bibitem{B93}
  J. Bourgain, {\it Fourier transform restriction phenomena for certain lattice subsets and applications to nonlinear evolution equations, I, Schr\"odinger equations}, Geom.\ Funct.\ Anal. {\bf 3} (1993), 107--156.

  \bibitem{BT08}
  N. Burq and N. Tzvetkov, {\it Random data Cauchy theory for supercritical wave equations I: Local theory}, Invent.\ Math. {\bf 173} (2008), 449--475.

  \bibitem{BT08-2}
  N. Burq and N. Tzvetkov, {\it Random data Cauchy theory for supercritical wave equations II: a global existence result}, Invent.\ Math. {\bf 173} (2008), 477--496.

  \bibitem{CG21}
  I. C\^impean and A. Grecu, {\it The nonlinear Schr\"odinger equation with white noise dispersion on quantum graphs}, Commun.\ Math.\ Sci. {\bf 19} (2021), 405--435.

  \bibitem{Gal23}
  L. Galeti, {\it Nonlinear Young Differential Equations: A Review}, J. Dynam.\ Differential Equations {\bf 35} (2023), 985--1046.

  \bibitem{CateGubi}
  R. Catellier and M. Gubinelli, {\it Averaging along irregular curves and regularisation of ODEs}, Stochastic Process.\ Appl. {\bf 126} (2016), 2323--2366.

  \bibitem{CG1}
  K. Chouk and M. Gubinelli, {\it Nonlinear PDEs with modulated dispersion I: Nonlinear Schr\"odinger equations}, Comm.\ Partial Differential Equations {\bf 40} (2015), 2047--2081.

  \bibitem{CG2}
  K. Chouk and M. Gubinelli, {\it Nonlinear PDEs with modulated dispersion II: Korteweg-de Vries equation}, arXiv:1406.7675.


  \bibitem{DT1}
  A. Debussche and Y. Tsutsumi, {\it 1D quintic nonlinear Schr\"odinger equation with white noise dispersion}, J. Math.\ Pures Appl. {\bf 96} (2011), 363--376.


  \bibitem{DNY3}
  Y. Deng, A. R. Nahmod and H. Yue, {\it Random tensors, propagation of randomness, and nonlinear dispersive equations}, Invent.\ Math. {\bf 228} (2022), 539--686.

  \bibitem{HL06}
  D. Hundertmark and Y.-R. Lee, {\it Decay estimates and smoothness for solutions of the dispersion managed non-linear Schr\"odinger equation}, Comm.\ Math.\ Phys. {\bf 286} (2009), 851--873.

  \bibitem{HL12}
  D. Hundertmark and Y.-R. Lee, {\it Super-exponential decay of diffraction managed solitons}, Comm.\ Math.\ Phys. {\bf 309} (2012), 1--21.

  \bibitem{K21}
  N. Kishimoto, {\it Unconditional local well-posedness for periodic NLS}, J. Differential Equations {\bf 274} (2021), 766--787.

  \bibitem{M06}
  R. Marty, {\it On a splitting scheme for the nonlinear Schr\"odinger equation in a random medium}, Commun. Math.\ Sci. {\bf 4} (2006), 679--705.

  \bibitem{Rob23}
  T. Robert,
  {\it Regularization by noise for some nonlinear dispersive PDEs},
  J.\ \'E.\ D.\ P. (2023), 12p.

  \bibitem{Ste23}
  G. Stewart, {\it On the wellposedness for periodic nonlinear Schr\"odinger equations with white noise dispersion}, Stoch.\ PDE: Anal.\ Comp. (2023).

  \bibitem{Tao}
  T. Tao, Nonlinear Dispersive Equations: Local and Global Analysis, CBMS, 2006.

  \bibitem{ZGTJH}
  V. Zharnitsky, E. Grenier, S. K. Turitsyn, C. K. R. T. Jones and J. S. Hesthaven, {\it Ground states of dispersion-managed nonlinear Schr\"odinger equation}, Phys.\ Rev.\ E (3) {\bf 62} (2000), 7358--7364.
\end{thebibliography}
\end{document}